\documentclass[12pt]{article}
\usepackage{graphicx,amsthm,amsmath,amsfonts,amssymb}
\usepackage[numbers]{natbib}
\usepackage{color}
\usepackage[]{hyperref}
\usepackage{ulem}

\newcommand{\ds}{\displaystyle }

\newcommand{\rs}[1]{{\mbox{\scriptsize \sc #1}}}
\newcommand{\vc}[1]{{\boldsymbol #1}} 

\newcommand{\sr}[1]{{\cal #1}}
 \newcommand{\dd}[1]{\mathbb{#1}}

 \newcommand{\rmn}[1]{\if#11I\else {\if#12I\hspace{-0.12ex}I\hspace{-0.85ex}\else {\if #13I\hspace{-0.16ex}I\hspace{-0.16ex}I\hspace{-1.6ex}\else I\hspace{-1.2ex}V \fi} \fi} \fi}

\newcommand{\eqn}[1]{(\ref{eqn:#1})}
\newcommand{\lem}[1]{Lemma~\ref{lem:#1}}
\newcommand{\cor}[1]{Corollary~\ref{cor:#1}}
\newcommand{\thr}[1]{Theorem~\ref{thr:#1}}

\newcommand{\app}[1]{Appendix~\ref{app:#1}}
\newcommand{\sectn}[1]{Section~\ref{sect:#1}}

\newcommand{\thrt}[1]{\ref{thr:#1}}

\newcommand{\sect}[1]{\ref{sect:#1}}

\newcommand{\br}[1]{\langle #1 \rangle}

\newcommand{\ol}{\overline}

\newcommand{\pend}{\hfill \thicklines \framebox(6.6,6.6)[l]{}}

\renewenvironment{proof}{\noindent {\sc  Proof.} \rm}{\pend}
\newenvironment{proof*}[1]{\noindent {\sc  #1} \rm}{\pend}

\newtheorem{theorem}{Theorem}[section]
\newtheorem{lemma}{Lemma}[section]

\newtheorem{remark}{Remark}[section]

\newtheorem{corollary}{Corollary}[section]

\newenvironment{mylist}[1]{\begin{list}{}
{\setlength{\itemindent}{#1mm}}
{\setlength{\itemsep}{0ex plus 0.2ex}}
{\setlength{\parsep}{0.5ex plus 0.2ex}}
{\setlength{\labelwidth}{10mm}}
}{\end{list}}

 \newcommand{\setnewcounter} {
 \setcounter{subsection}{0}
 \setcounter{equation}{0}
 \setcounter{conjecture}{0}
 \setcounter{assumption}{0}
 \setcounter{question}{0}
 \setcounter{definition}{0}
 \setcounter{theorem}{0}
 \setcounter{corollary}{0}
 \setcounter{lemma}{0}
 \setcounter{proposition}{0}
 \setcounter{remark}{0}
}

\begin{document}

\title{Two-node fluid network with 
a heavy-tailed random input: the 
strong stability case}

\author{Sergey Foss\\Heriot Watt University and\\Institute of Mathematics, Novosibirsk\\ \and Masakiyo Miyazawa\\ Tokyo University of Science}

\date{}

\maketitle

\begin{abstract}
We consider a two-node fluid network with batch arrivals  
of random size having a heavy-tailed distribution. 
We are interested in the tail asymptotics for the stationary distribution of a two-dimensional queue-length
process. The tail asymptotics have been well studied for two-dimensional reflecting processes where jumps have either a bounded or an unbounded light-tailed distribution. However, presence of heavy tails totally changes the asymptotics. Here we focus on the case of {\it strong
stability} 
where both nodes release fluid with sufficiently high speeds to minimise their mutual influence.
We show that, like in the one-dimensional case, big jumps provide the main cause for queues to become large, but now they may have multidimensional features. For deriving these results, we develop an analytic approach that 
differs from the traditional tail asymptotic studies, and obtain various weak tail equivalences. 
Then, in the case of one-dimensional subexponential jump-size distributions, we find the exact
asymptotics based on  
the sample-path
arguments.
\end{abstract}

\section{Introduction}
\label{sect:introduction}

Tail asymptotics problems have been studied in queueing networks 
and related reflecting processes for long years, but new 
developments are still going on. A key feature of that is an influence of the multiple boundary faces in a multidimensional 
state space. It requires an analysis that differs from the traditional one. Recent studies of those multidimensional processes have been mainly done in the {\it light-tail regime} 
where no heavy tails arise (see, e.g., \cite{Miya2011}). On the other hand, the heavy-tail asymptotics are mostly studied for processes with single boundary faces or for certain 
monotone characteristics (see, e.g., \cite{FossKorsZach2011} or \cite{BaccFoss2004}).

Thus, it is natural to ask how presence of heavy tails changes the tail asymptotics in multidimensional reflecting processes including queueing networks. The aim of this paper is to analyse this problem for the stationary distribution of a continuous-time reflecting process in the two-dimensional nonnegative quadrant. For this, we consider a two-node fluid network with a compound input with either Poisson or renewal arrivals, which is a simple model but still keeps the feature of a multidimensional reflecting process. 
It may be viewed as a continuous-time approximation of a generalised Jackson 
network with 
simultaneous arrivals of big batches of customers.

We analyse the tail asymptotics for this fluid network as follows.  First we assume Poisson arrivals and develop the analytic 
approach based on the stationary (balance) equation. We obtain lower and upper bounds for the stationary distributions in the coordinate and arbitrary directions 
(Theorems \thrt{option 0 1} and \thrt{bound 0a}) and deduce the weak tail equivalence
of those bounds under subexponentiality assumptions. Then we turn to the sample-path approach. We assume renewal arrivals and obtain a lower bound based on the 
long-tailedness of the jumps distributions and, then, under subexponentiality,
show that the lower bound provides the exact asymptotics.   
 
Our results relate to the tail asymptotics in generalised Jackson networks (and in more general classes of max-plus systems) with heavy-tailed distributions of service times that
have been studied in
 \cite{BaccFoss2004} (see also \cite{BaccFossLela2005,Lela2005a}). There
 the exact tail asymptotics was found only for the ``maximal data'' (time needed to empty the system in the steady state after stopping the input process). In tandem queues, the maximal data coincides with the stationary sojourn time of a ``typical'' customer. But the two notions differ when routing includes feedbacks. Another novel element of the paper is in considering a multidimensional heavy-tailed input, with possible dependence between its coordinates.

The paper is organised as follows. In Section \ref{sect:fluid}, we introduce  
a fluid network with random jumps and discuss its dynamics. 
Then Sections \sect{analytic 1} and \sect{analytic 2} deal with the analytic approach, and \sectn{sample} with the sample-path analysis. Appendix contains an auxiliary material that includes the analysis of the corresponding fluid model and basic definitions and properties of subexponential distributions.

\section{Fluid network with compound input}  \setnewcounter
\label{sect:fluid}

Consider a two-node fluid network where nodes $i=1,2$ receive an input process $\vc{\Lambda}(t)=(\Lambda_{1}(t),\Lambda_{2}(t))$, which is a compound process generated by point process $\{N(t); t \ge 0\}$ and i.i.d.\ jumps $\{(J_{1,n},J_{2,n}); n=1,2,\ldots\}$. In this Section, we only  assume that $\{N(s+t); t \ge 0 \}$ weakly converges to a stationary point process $\{N^{*}(t); t \in \dd{R} \}$ as $s \downarrow -\infty$ and its intensity $\lambda \equiv \dd{E}(N^{*}(1))$ is finite. In the subsequent sections, we will specify $\{N(t)\}$ to be a Poisson process for the analytic approach, and a renewal process for the sample-path approach. In what follows, we use a shorter 
notation $\vc{J} \equiv (J_{1},J_{2})$ for a random vector having the same distribution with 
$(J_{1,n},J_{2,n})$. The joint distribution of $(J_{1},J_{2})$ is denoted by $F$, with marginals $F_i$.     
We let $m_i= \dd{E} J_i$ and 
$\alpha_{i} = \lambda m_{i}$ for $i=1,2$.

Both nodes have infinite capacity buffers and release fluid with corresponding rates 
$\mu_{i}$, $i=1,2$. The $p_{ij}$ proportion of the outflow from node $i$ goes to node $j$ for $i,j=1,2$, while the remaining proportion $1-p_{ij}$ leaves the system. We assume that 
\begin{eqnarray}
\label{eqn:p condition}
  0 \le p_{12} p_{21} < 1, \quad 0 < p_{12} + p_{21}
\end{eqnarray}
to exclude the trivial boundary cases including parallel queues, $p_{12}=p_{21}=0$. Without loss of generality, we may also assume that $p_{ii}=0$, for $i=1.2$.

\begin{remark}
\label{rem:ext} 
One may assume that,
in addition to the jump input, there are continuous fluid inputs
to both queues, say with rates $\beta_{1}$ and $\beta_{2}$ respectively.
Given stability, such a model may be reduced to the original one,
by slowing down the release rates, namely, by replacing
$\mu_i$, $i=1,2$ with $\mu_i 
(1-(\beta_i+\beta_{3-i}p_{3-i,i})/(1-p_{12}p_{21}))$. 
\end{remark}

  We now introduce a buffer content process $\vc{Z}(t) \equiv (Z_{1}(t), Z_{2}(t))^{\rs{t}}$, which is defined as a nonnegative solution to the following equations:
\begin{eqnarray}
\label{eqn:Z1a}
  && Z_{1}(t) = Z_{1}(0) + \Lambda_{1}(t) +p_{21} (\mu_{2}t - Y_{2}(t)) - \mu_{1} t + Y_{1}(t),\\
\label{eqn:Z2a}
  && Z_{2}(t) = Z_{2}(0)+ \Lambda_{2}(t) +p_{12} (\mu_{1}t - Y_{1}(t)) - \mu_{2} t + Y_{2}(t),
\end{eqnarray}
  where $Y_{i}(t)$ is the minimal nondecreasing process that keeps $Z_{i}(t)$ to stay nonnegative. As usual, we assume that sample paths are right-continuous and have left-hand limits.
    
Let $\delta_{1} = \mu_{1} - \mu_{2} p_{21}$ and $\delta_{2} = \mu_{2} - \mu_{1} p_{12}$. 
Put $ X_{i}(t) 
=\Lambda_{i}(t) -  \delta_{i} t$ and let $\vc{X}(t) = (X_{1}(t), X_{2}(2))^{\rs{t}}$, $\vc{Y}(t) = (Y_{1}(t), Y_{2}(t))^{\rs{t}}$, and
\begin{eqnarray*}
  R = \left(\begin{array}{cc}1 & - p_{21} \\-p_{12} & 1\end{array}\right).
\end{eqnarray*}
Then \eqn{Z1a} and \eqn{Z2a} may be rewritten as
\begin{eqnarray}
\label{eqn:reflecting Jump}
  \vc{Z}(t) = \vc{Z}(0) + \vc{X}(t) + R \vc{Y}(t), \qquad t \ge 0.
\end{eqnarray}
This is the standard definition of a reflecting process (in the nonnegative quadrant $\dd{R}_{+}^{2}$) for a given process $\vc{X}(t)$, where $\vc{Y}(t)$ is a {\it regulator} such that $Y_{i}(t)$ increases only when $Z_{i}(t) = 0$. Here $R$ is called a {\it reflection matrix} (e.g., see Section 3.5 of \cite{Miya2011}). By \eqn{p condition}, the inverse $R^{-1}$ exists and is nonnegative. This guarantees the existence of $\{\vc{Z}(t); t \ge 0\}$. We refer to this process 
as to a {\it two-dimensional fluid network with compound inputs}.

  Since $R^{-1} \vc{\alpha}$ is the total inflow rate vector, the fluid network is stable if and only if
\begin{equation}
\label{eqn:stability 1}
  R^{-1} \vc{\alpha} < \vc{\mu} 
\end{equation}
where the inequality is strict in both coordinates. A formal proof for this stability condition can be found in \cite{Kell1996}. Let $\Delta_{i} = - \dd{E}(X_{i}(1))$. 
Then $\Delta_{i} = \delta_{i} - \alpha_{i}$, and 
  stability condition \eqn{stability 1} is equivalent to
\begin{equation}
\label{eqn:stability 2}
  \Delta_{1} + \Delta_{2} p_{21} > 0, \qquad \Delta_{1} p_{12} + \Delta_{2} > 0.
\end{equation}
Under condition \eqref{eqn:stability 1}, the stationary distribution, say $\pi$, of $\vc{Z}(t)$ uniquely exists. Let $\vc{Z} \equiv (Z_{1}, Z_{2})$ be a random vector subject to $\pi$.

We are interested in the tail behaviour of $\dd{P}(c_{1}Z_{1} + c_{2}Z_{2} > x)$ as $x$ goes to infinity, for a given {\it directional} vector $\vc{c} \equiv (c_{1}, c_{2}) \ge \vc{0}$ satisfying $c_{1} + c_{2} = 1$. In this paper, we consider this asymptotic mostly under the {\it strong stability} condition, that is
\begin{eqnarray}
\label{eqn:strong stability}
  \Delta_{1} > 0 \qquad \mbox{and} \qquad \Delta_{2} > 0.
\end{eqnarray}
 Other cases will be studied in a companion paper \cite{FossMiya2013}. Under \eqref{eqn:strong stability}, both nodes are sufficiently fast to process fluids given the input is always maximal,
 and the following holds.
\begin{lemma}
\label{lem:upper_parallel}

(a, Sample-path majorant). On any elementary event, consider an auxiliary model of two parallel queues, with node 
$i=1,2$ 
having a continuous input of rate $\mu_{3-i}p_{3-i,i}$, release rate $\mu_i$, and jump input
process $\Lambda_i$. Let $\tilde{Z}_{i}(t)$ be the content of node $i$ at time $t$.
If $\tilde{Z}_{i}(0)\ge Z_{i}(0)$, then $\tilde{Z}_{i}(t)\ge Z_{i}(t)$, for any $t$.

(b, Stable majorant). Assume  that the input is  a renewal process and \eqref{eqn:strong stability} holds. Then the processes $\tilde{Z}_{i}(t)$ admit a unique stationary version,
and,.
under the natural coupling of the input processes, $\tilde{Z}_{i} 
\ge {Z}_i$ a.s.
\end{lemma}
{\sc Proof}. Between any two jumps, the trajectories of $Z_{i}(t)$ and $\tilde{Z}_{i}(t)$
are Lipschitz, and at any regular point $t$ with $Z_{i}(t)>0$, $\tilde{Z}_{i}(t)>0$, the derivative of $Z_{i}$ is smaller than that of $\tilde{Z}_{i}$. So inequality $\widetilde{Z}_{i}(t)\ge Z_{i}(t)$ is preserved between any two jumps. Since the jumps are synchronous and the jump
sizes are equal, the induction argument completes the proof of (a). Then statement (b) is straightforward.

\section{Analytic approach: Basic tools} \setnewcounter
\label{sect:analytic 1}

In this section, we assume that
\begin{itemize}
\item [(A1)] $\{\Lambda(t)\}$ is a compound Poisson process with rate $\lambda$,
\end{itemize}
and derive decomposition formulae for the stationary distribution in terms of moment generating functions. We first obtain the stationary balance equation under the stability condition \eqn{stability 2}.

Let $C^{1}(\dd{R}^{2})$ be the set of all functions from $\dd{R}^{2}$ to $\dd{R}$ having continuous first order partial derivatives. We write $f'_{i}(x_{1},x_{2})$ instead of $\frac {\partial}{\partial x_{i}} f(x_{1},x_{2})$, for short. It follows from \eqn{reflecting Jump} that, for $f \in C^{1}(\dd{R}^{2})$, the increment $f(\vc{Z}(1)) - f(\vc{Z}(0))$ can be expressed as integrations on $[0,1]$ with respect to $dt$, $d\Lambda(t)$ and $dY_{i}(t)$ (formally by It\^o's integral formula). Then, taking the expectations with $\vc{Z}(0)$ subject to the stationary distribution $\pi$ (we denote this expectation by $\dd{E}_{\pi}$) and recalling the stationary version $\vc{Z} \equiv (Z_{1}, Z_{2})$, we have
\begin{eqnarray}
\label{eqn:stationary 1}
  \lefteqn{\sum_{i=1}^{2} \left( - \delta_{i} \dd{E}_{\pi}(f'_{i}(\vc{Z}))\right) + \lambda \dd{E}_{\pi} (f(\vc{Z} + \vc{J}) - f(\vc{Z})) } \nonumber \\
  && + \dd{E}_{1} (f'_{1}(0, Z_{2}) - p_{12} f'_{2}(0, Z_{2})) + \dd{E}_{2} (- p_{21} f'_{1}(Z_{1}, 0) + f'_{2}(Z_{1}, 0)) = 0, \qquad
\end{eqnarray}
as long as all expectations are finite, where jump size vector $\vc{J}$ is independent of everything else. Here $\dd{E}_{i}$ represents the expectation with respect to the Palm measure concerning $\{Y_{i}(t)\}$, that is, for any bounded measurable function $g$ on $\dd{R}_{+}$,
\begin{eqnarray*}
  \dd{E}_{i}(g(Z_{3-i})) = \dd{E}_{\pi} \left( \int_{0}^{1} g(Z_{3-i}) Y_{i}(du) \right), \qquad i = 1,2.
\end{eqnarray*}
These expectations uniquely determine finite
measures $\nu_{i}$ on $(\dd{R}_{+}, \sr{B}(\dd{R}_{+}))$, where $\sr{B}(\dd{R}_{+})$ is the Borel 
$\sigma$-field on $\dd{R}_{+}$. They are called {\it boundary measures}. We denote a random variable with probability distribution $\frac 1{\nu(\dd{R}_{+})} \nu_{3-i}$ by $V_i$.

The stationary equation \eqn{stationary 1} uniquely determines the stationary distribution $\pi$ if it holds for a sufficiently large class of function $f$. For this, we may choose a class of exponential functions $f(\vc{x}) = e^{\br{\vc{\theta}, \vc{x}}}$ on $\dd{R}_{+}^{2}$ for each $\vc{\theta} \equiv (\theta_{1}, \theta_{2}) \le \vc{0}$, where $\br{\vc{a}, \vc{b}}$ stands for the inner product of vectors 
$\vc{a}, \vc{b} \in \dd{R}^{2}$. Let $\vc{\delta} = (\delta_{1}, \delta_{2})$, and let
\begin{eqnarray*}
 && \varphi(\vc{\theta}) = \dd{E}_{\pi}(e^{\br{\vc{\theta}, \vc{Z}}}), \qquad  \varphi_{3-i}(\theta_{i}) = \dd{E}_{3-i}(e^{\theta_{i} Z_{i}}), \qquad i = 1,2,\\
  && \widehat{F}(\vc{\theta}) = \dd{E} (e^{\br{\vc{\theta},\vc{J}}}), \qquad \kappa(\vc{\theta}) = \br{\vc{\delta}, \vc{\theta}} - \lambda (\widehat{F}(\vc{\theta}) - 1).
\end{eqnarray*}
Here $-\kappa(\vc{\theta})$ is the {\it  L\'{e}vy component} of $\vc{X}(t)$. Then, \eqn{stationary 1} becomes
\begin{eqnarray}
\label{eqn:stationary 2}
  \kappa(\vc{\theta}) \varphi(\vc{\theta}) = (\theta_{1} - p_{12} \theta_{2}) \varphi_{1}(\theta_{2}) + (\theta_{2}-p_{21}\theta_{1} ) \varphi_{2}(\theta_{1}),
\end{eqnarray}
as long as $\varphi(\vc{\theta})$, $\widehat{F}(\vc{\theta})$ and $\varphi_{i}(\theta_{i})$ are finite. Clearly, \eqn{stationary 2} is always valid for $\vc{\theta} \le \vc{0}$.

For convenience of computations in subsequent sections, we find first $\nu_{i}(\{0\})$ and $\nu_{i}(\dd{R}_{+})$. Clearly, they are identical with $\varphi_{i}(-\infty)$ and $\varphi_{i}(0)$, respectively.

\begin{lemma} 
\label{lem:boundary measures 1}
Under the stability condition \eqn{stability 2}, for $\vc{\theta} \le 0$, 
\begin{equation}
\label{eqn:boundary 1a}
 \varphi_{1}(\theta_{2}) = \delta_{1} \varphi(-\infty,\theta_{2}) + p_{21} \varphi_{2}(-\infty)  
\quad \mbox{and}
\quad 
\varphi_{2}(\theta_{1}) = \delta_{2} \varphi(\theta_{1}, -\infty) + p_{12} \varphi_{1}(-\infty).
\end{equation}
\end{lemma}
\begin{proof}
   Dividing \eqn{stationary 2} by $\theta_{1}$, we have
\begin{eqnarray}
\label{eqn:stationary 3}
  \left( \delta_{1} + \delta_{2} \frac {\theta_{2}}{\theta_{1}} - \frac {\lambda} {\theta_{1}} (\widehat{F}(\vc{\theta}) - 1) \right) \varphi(\vc{\theta}) = \left(1 - p_{12} \frac {\theta_{2}}{\theta_{1}} \right) \varphi_{1}(\theta_{2}) + \left(\frac {\theta_{2}}{\theta_{1}} -p_{21} \right) \varphi_{2}(\theta_{1}).
\end{eqnarray}
  Letting $\theta_{1}\to -\infty$, 
  we get the first equality in 
  \eqn{boundary 1a}. The symmetry gives the 
  second. 
  \end{proof}

  Denote the traffic intensities at nodes 1 and 2, respectively, by
\begin{eqnarray*}
  \rho_{1} = 
  ({\alpha_{1} + \alpha_{2} p_{21}})/({\mu_{1}(1-p_{12} p_{21})}), \qquad 
  \rho_{2} = 
  ({\alpha_{2} + \alpha_{1} p_{12}})/({\mu_{2}(1-p_{12} p_{21})}).
\end{eqnarray*}
  
\begin{lemma}
\label{lem:boundary measures 2}
Under the stability condition \eqn{stability 2}, for $i=1,2$,
\begin{eqnarray}
\label{eqn:phi 0}
 \varphi_{i}(-\infty) = \mu_{i} \pi(\vc{0}), \qquad
 \varphi_{i}(0) = \frac {\Delta_{i} + \Delta_{3-i} p_{(3-i)i}} {1-p_{12} p_{21}} = \mu_{i}(1-\rho_{i}).
\end{eqnarray}
In particular, $\pi(\vc{0}) = \dd{P}(\vc{Z} = \vc{0}) < \min(1-\rho_{1}, 1-\rho_{2})$, and, for $i=1,2$,
\begin{eqnarray}
\label{eqn:V1 Z}
 && \dd{P}(V_{i} > x) = \frac {\delta_{3-i}} {\mu_{3-i}(1-\rho_{3-i})} \dd{P}_{\pi}(Z_{i} > x, Z_{3-i} = 0), \qquad x \ge 0.
\end{eqnarray}
\end{lemma}

\begin{proof}
Letting $\theta_{1}$ and $\theta_{2}$ to $-\infty$ in \eqn{boundary 1a}, 
we get
$$\varphi_{1}(-\infty) = \delta_{1} \varphi(-\infty,-\infty) + p_{21} \varphi_{2}(-\infty) ,
\quad
\varphi_{2}(-\infty) = \delta_{2} \varphi(-\infty, -\infty) + p_{12} \varphi_{1}(-\infty).
$$
Since $\varphi(-\infty,-\infty) = \pi(\vc{0})$, solving these equations yields $\varphi_{i}(-\infty) = \mu_{i} \pi(\vc{0})$. On the other hand, with putting $\theta_{2} = 0$,
letting $\theta_{1} \to -\infty$ in \eqn{stationary 3}, and using the symmetry, we get 
\begin{eqnarray*}
  \Delta_{1} + p_{21} \varphi_{2}(0) = \varphi_{1}(0), \qquad \Delta_{2} + p_{12} \varphi_{1}(0) = \varphi_{2}(0).
\end{eqnarray*}
Solving these equation yields the first equality of \eqn{phi 0} for $\varphi_{i}(0)$. The second equality is immediate from the definitions of $\Delta_{i}$ and $\rho_{i}$.
\end{proof}

From equations \eqn{phi 0}, it is easy to see that $(\varphi_{1}(0), \varphi_{2}(0))^{\rs{t}} = \vc{\mu} - R^{-1} \vc{\alpha} > \vc{0}$. Also, the second equality for $\varphi_{i}(0)$ yields another representation for $\Delta_{i}$:
\begin{eqnarray}
\label{eqn:Delta 2}
  \Delta_{1} = \mu_{1}(1-\rho_{1}) - \mu_{2} p_{21}(1 - \rho_{2}), \qquad \Delta_{2} = \mu_{2}(1-\rho_{2}) - \mu_{1} p_{12}(1 - \rho_{1}).
\end{eqnarray}

We now turn to the tail probabilities. To this end, we single out the moment generation function $\varphi(s\vc{c})$ of one-dimensional random variable $c_{1}Z_{1} + c_{2}Z_{2}$ from the the stationary equation \eqn{stationary 2} or \eqn{stationary 3}.
Namely, we 
(i) derive $\varphi(s\vc{c})$ as a linear combination of the moment generating functions of 
certain measures, which may include unknown boundary measures $\nu_{1}$ and $\nu_{2}$; 
(ii) find lower and upper bounds for the tail probability $\dd{P}_{\pi}(c_{1}Z_{1} + c_{2}Z_{2} > x)$ from the expression obtained in (i);
then (iii) obtain the asymptotics for $\dd{P}_{\pi}(c_{1}Z_{1} + c_{2}Z_{2} > x)$ as $x \to \infty$ using the heavy-tailedness of jump size distributions.

Clearly, the most important step is (i). Our arguments are similar
to that in deriving the Pollaczek-Khinchine formula from the stationary equation of the $M/G/1$ queue. However, we have to be careful because here the reflecting process is two-dimensional, while the boundary is a single point in the case of the $M/G/1$ queue. Also our analysis
depends 
on direction $\vc{c}$.

We first write the stationary equation \eqn{stationary 2} with $\vc{\theta} = s \vc{c}$ for a directional vector $\vc{c}$ as
\begin{eqnarray}
\label{eqn:stationary 4}
  \left(c_{1} \delta_{1} + c_{2}\delta_{2} - \frac {\lambda} {s} (\widehat{F}(s\vc{c}) - 1)\right) \varphi(s\vc{c}) = (c_{1} - p_{12} c_{2}) \varphi_{1}(c_{2} s) + (c_{2} - p_{21} c_{1}) \varphi_{2}(c_{1} s) . \quad
\end{eqnarray}
To express the coefficient of $\varphi(s\vc{c})$ in a compact form, we introduce the integrated probability distribution $F^{I}_{\vc{c}}$ by 
\begin{eqnarray*}
F^{I}_{\vc{c}}(x) = 1-\frac{1}{m_{\vc{c}}} \int_x^{\infty} \dd{P} (c_{1} J_{1} + c_{2} J_{2}>y) dy, \quad
x\ge 0,
\end{eqnarray*}
where $m_{\vc{c}} = c_{1} m_{1} + c_{2} m_{2}$, and denote its moment generating function by $\widehat{F}^{I}_{\vc{c}}$. Namely, $\widehat{F}^{I}_{\vc{c}}(s) = (\widehat{F}(s \vc{c}) - 1)/(m_{\vc{c}} s)$. For positive $r < 1$, let
\begin{eqnarray*}
  \widehat{S}^{I(r)}_{\vc{c}}(s) = (1-r) (1 - r \widehat{F}^{I}_{\vc{c}}(s))^{-1},
\end{eqnarray*}
be the moment generating function of the geometric sum with parameter $r$ of $i.i.d.$ random variables having distribution $F_{\vc{c}}^I$. 

By the strong stability assumption \eqn{strong stability}, $c_{2} \delta_{1} + c_{2} \delta_{2} - (c_{1} \alpha_{1} + c_{2} \alpha_{2} ) = c_{1} \Delta_{1} + c_{2} \Delta_{2} > 0$, and therefore
\begin{eqnarray*}
  r_{\vc{c}} 
  := \frac {c_{1} \alpha_{1} + c_{2} \alpha_{2}} {c_{1} \delta_{1} + c_{2} \delta_{2}} < 1.
\end{eqnarray*}
Since $\lambda m_{\vc{c}} = c_{1} \alpha_{1} + c_{2} \alpha_{2}$, \eqn{stationary 4} can be rewritten as
\begin{eqnarray}
\label{eqn:stationary 6}
  \varphi(s\vc{c}) =\left(  (c_{1} - p_{12} c_{2}) \varphi_{1}(c_{2} s) + (c_{2} - p_{21} c_{1}) \varphi_{2}(c_{1} s) \right) \widehat{S}^{I(r_{\vc{c}})}_{\vc{c}}(s).
\end{eqnarray}
Thus, if the coefficients of $\varphi_{1}(c_{2} s)$ and $\varphi_{2}(c_{1} s)$ are positive, then we have a decomposition for the distribution of $c_{1} Z_{1} + c_{2} Z_{2}$. However, those coefficients may be negative. 

We note that, if both of $c_{1} - p_{12} c_{2} $ and $c_{2} - p_{21} c_{1} $ are non-positive, then $c_{1} (1- p_{12} p_{21}) \le 0$ and $c_{2} (1- p_{12} p_{21}) \le 0$, which contradict \eqn{p condition} and $\vc{c} \ne \vc{0}$. Hence, either one of them is positive at least, and the following three cases are only possible under \eqn{strong stability}. 
\begin{mylist}{5}
\item [(C0)] $c_{1} - p_{12} c_{2} \ge 0$ and $c_{2} - p_{21} c_{1} \ge 0$ (in this case, we must have $\vc{c} > \vc{0}$).
\item [(C1)] $c_{1} - p_{12} c_{2} \ge 0$ and $c_{2} - p_{21} c_{1} < 0$, \hspace{2ex} (C2) $c_{1} - p_{12} c_{2} < 0$ and $c_{2} - p_{21} c_{1} \ge 0$.
\end{mylist}
Since (C1) and (C2) are symmetric, we consider only (C0) and (C1).

Recall that $V_{1}$ and $V_{2}$ have the probability distributions normalized by $\nu_{1}$ and $\nu_{2}$, respectively. Then, from \eqn{stationary 6}, we have the following lemma.

\begin{lemma}
\label{lem:decomposition 1}
For directional vector $\vc{c} \ge \vc{0}$, we have, for $B \in \sr{B}(\dd{R}_{+})$,
\begin{eqnarray}
\label{eqn:decomposition 1}
  \lefteqn{\dd{P}_{\pi}(c_{1} Z_{1} + c_{2} Z_{2} \in B)} \nonumber\\
  && = \eta^{(1)}_{\vc{c}} \mu_{1} (1-\rho_{1}) \dd{P}(c_{2} V_{2} + S_{\vc{c}}^{I(r_{\vc{c}})} \in B) + \eta^{(2)}_{\vc{c}} \mu_{2} (1-\rho_{2}) \dd{P}(c_{1} V_{1} + S_{\vc{c}}^{I(r_{\vc{c}})} \in B) , \qquad
\end{eqnarray}
  where $V_{1}$ and $V_{2}$ are independent of $S_{\vc{c}}^{I(r_{\vc{c}})}$, and 
\begin{eqnarray}
\label{eqn:eta}
\eta^{(1)}_{\vc{c}} = \frac {c_{1} - p_{12} c_{2}} {c_{1} \Delta_{1} + c_{2} \Delta_{2}} ,
\qquad
\eta^{(2)}_{\vc{c}} = \frac {c_{2} - p_{21} c_{1}} {c_{1} \Delta_{1} + c_{2} \Delta_{2}} .
  \end{eqnarray}
\end{lemma}

Since $\eta^{(1)}_{\vc{c}} \mu_{1} (1-\rho_{1}) + \eta^{(2)}_{\vc{c}} \mu_{1} (1-\rho_{1}) =1$ from \eqn{decomposition 1} with $B = \dd{R}_{+}$ and $\eta^{(i)}_{\vc{c}}$ for $i=1,2$ are positive for (C0), we get the following lower bound.

\begin{corollary}
\label{cor:lower bound 1}
For the case (C0), we have
\begin{equation}
\label{eqn:LB++}
\dd{P}(S_{\vc{c}}^{I(r_{\vc{c}})} > x) \le \dd{P}_{\pi}(c_{1} Z_{1} + c_{2} Z_{2} > x) \qquad
\mbox{for all } x >0.
\end{equation}
\end{corollary}

For the case (C1), \eqn{decomposition 1} cannot be used to get a lower bound, and we use another expression. Let
$\varphi^{+}(\vc{\theta}) = \dd{E}(e^{\br{\vc{\theta}, \vc{Z}}} 1(\vc{Z} > \vc{0}))$ and $d_{0} = \frac 1{\delta_{1} \delta_{2}} \mu_{1} \mu_{2} (1-p_{12}p_{21}) \pi(\vc{0})$.
From \lem{boundary measures 1}, we have
\begin{eqnarray}
\label{eqn:another 1}
  \varphi(s\vc{c}) = \varphi^{+}(s\vc{c}) + \frac {\varphi_{2}(c_{1}s)} {\delta_{2}} + \frac {\varphi_{1}(c_{2}s)} {\delta_{1}} - d_{0}
\end{eqnarray}
where we have used the fact that
\begin{eqnarray*}
  \delta_{1} p_{12} \varphi_{1}(-\infty) + \delta_{2} p_{21} \varphi_{2}(-\infty) + \delta_{1} \delta_{2} \pi(\vc{0})
 = \mu_{1} \mu_{2} (1-p_{12}p_{21}) \pi(\vc{0}) = \delta_{1} \delta_{2} d_{0}. \quad
\end{eqnarray*}

\begin{lemma}
\label{lem:decomposition 2}
  For the case (C1), let $\ds r'_{\vc{c}} = \frac {c_{1} \alpha_{1} + c_{2} \alpha_{2}} {c_{1} (\delta_{1} + \delta_{2} p_{21})}$, then $0< r'_{\vc{c}} < 1$ and
\begin{eqnarray}
\label{eqn:decomposition 2a}
  \varphi(s\vc{c}) = \Big(\frac {d^{(1)}_{\vc{c}}}{\delta_{1}} \varphi_{1}(c_{2} s) + d^{(2)}_{\vc{c}} \left(\varphi^{+}(s\vc{c}) - d_{0}\right)\Big) \widehat{S}_{\vc{c}}^{I(r'_{\vc{c}})}(s),
\end{eqnarray}
where
\begin{eqnarray*}
 d^{(1)}_{\vc{c}} = \frac {\delta_{1}(c_{1} - p_{12} c_{2}) + \delta_{2} (p_{21} c_{1} - c_{2})} {c_{1} (\delta_{1} + \delta_{2} p_{21})(1-r'_{\vc{c}})}, \quad
 d^{(2)}_{\vc{c}} = \frac {\delta_{2}(p_{21} c_{1} - c_{2})} {c_{1} (\delta_{1} + \delta_{2} p_{21})(1-r'_{\vc{c}})}.
\end{eqnarray*}
Therefore
\begin{eqnarray}
\label{eqn:decomposition 2b}
  \dd{P}(S_{\vc{c}}^{I(r'_{\vc{c}})} > x) \le \dd{P}_{\pi}(c_{1} Z_{1} + c_{2} Z_{2} > x), \qquad x \ge 0.
\end{eqnarray}
\end{lemma}
\begin{proof}
Multiplying $(p_{21} c_{1} - c_{2}) \delta_{2}$ with \eqn{another 1} and adding to \eqn{stationary 6}, we have
 \begin{eqnarray*}
  \lefteqn{\left(c_{1} (\delta_{1} + \delta_{2} p_{21}) - (c_{1} \alpha_{1} + c_{2} \alpha_{2}) \widehat{F}^{I}_{\vc{c}}(s) \right) \varphi(s\vc{c})} \nonumber \hspace{0ex}\\
  && = \frac 1{\delta_{1}} (\delta_{1} (c_{1} - p_{12} c_{2}) + \delta_{2} (p_{21} c_{1} - c_{2})) \varphi_{1}(c_{2} s) + \delta_{2} (p_{21} c_{1} - c_{2}) \left(\varphi^{+}(s\vc{c}) - d_{0}\right). \quad
\end{eqnarray*}
This yields \eqn{decomposition 2a} because $c_{1} (\delta_{1} + \delta_{2} p_{21}) - (c_{1} \alpha_{1} + c_{2} \alpha_{2}) = c_{1} (\Delta_{1} +\Delta_{2} p_{21}) + \alpha_{2} (c_{1} p_{21} - c_{2}) > 0$. Since $d^{(1)}_{\vc{c}} > 0$ and $d^{(2)}_{\vc{c}} > 0$, the right-hand side of \eqn{decomposition 2a} represents the convolutions of two distributions on $[0,\infty)$, and this leads to \eqn{decomposition 2b}.
 \end{proof}

\section{Analytic approach: Bounds and tail asymptotics} \setnewcounter
\label{sect:analytic 2}

We continue to assume (A1), and consider the tail probability $\dd{P}_{\pi}(c_{1} Z_{1} + c_{2} Z_{2} > x)$ for 
directional vector $\vc{c} \ge \vc{0}$. We start with $\vc{c} = (1,0)^{\rs{t}}$ where
results are obtained under weaker assumptions.

\begin{theorem}
\label{thr:option 0 1} Assume (A1) and that the system is stable and that 
  $\Delta_{1} > 0$. 
  We have
\begin{eqnarray}
\label{eqn:tail Z1a}
  \dd{P}(S_{1}^{I(r'_{1})} > x) \le \dd{P}_{\pi}(Z_{1} > x) \le \dd{P}(S_{1}^{I(r_{1})} > x), \qquad x > 0,
\end{eqnarray}
where $r_{1} = \alpha_{1}/\delta_{1}$ and $r'_{1} = \alpha_{1}/(\delta_{1} + \delta_{2} p_{21})$.
\end{theorem}
\begin{proof}
  The upper bound in \eqn{tail Z1a} is immediate from \lem{upper_parallel} because $\tilde{Z}_{1}$ is subject to the stationary workload distribution of the $M/G/1$ queue. It also can be analytically obtained from \eqn{stationary 6}. The lower bound is already obtained in \lem{decomposition 2}.
\end{proof}

\begin{remark}
\label{rem:option 0 1}
  Clearly, $\Delta_{1} > 0$ and $\mu_{2} - \mu_{1}p_{12} = \delta_{2} > 0$ imply that
\begin{eqnarray*}
  \frac {\alpha_{1}} {\mu_{1}} < r'_{1} = \frac {\alpha_{1}} {\mu_{1}(1- p_{12} p_{21})} < r_{1} = \frac {\alpha_{1}}{\mu_{1}  - \mu_{2} p_{21}} < 1.
\end{eqnarray*}
   By arguments similar to that in \lem{upper_parallel}, $\dd{P}(S_{1}^{I(p_{1})} > x)$ with $p_{1} = \alpha_{1}/\mu_{1}$ also gives a lower bound, but the lower bound in \eqn{tail Z1a} is tighter than 
   that because $p_{1} < r'_{1}$.
\end{remark}

To compare tails of distributions, we recall the following notion.
Distribution functions $F$ and $G$ are {\it weakly tail-equivalent} if,
for $\overline{F}(x)=1-F(x)$ and $\overline{G}(x)=1-G(x)$, 
$$
  0 < \liminf_{x \to \infty} \overline{F}(x)/\overline{G}(x)  \le
  \limsup_{x \to \infty} \overline{F}(x)/\overline{G}(x) < \infty, 
$$
see \cite{FossKorsZach2011} for details and related topics. By Property 4 in \app{heavy tail}, as $x \to \infty$,
\begin{eqnarray}
\label{eqn:SI asymptotics}
 \dd{P}(S_{1}^{I(r_{1})} > x) \sim \frac {r_{1}} {1-r_{1}} \ol{F}^{I}_{1}(x), \qquad \dd{P}(S_{1}^{I(r'_{1})} > x) \sim \frac {r'_{1}} {1-r'_{1}}  \ol{F}^{I}_{1}(x).
\end{eqnarray}
Hence, \thr{option 0 1} yields
\begin{corollary}
\label{cor:weakly equivalent 1}
  Under the assumptions of \thr{option 0 1}, if $F_{1}^{I}$ is subexponential, then the distribution of $Z_{1}$ is weakly tail-equivalent to $F_{1}^{I}$.
\end{corollary}

We next consider the case of directional vector $\vc{c} > \vc{0}$ for options (C0) and (C1).

\begin{theorem}
\label{thr:bound 0a}
Assume (A1) and the strong stability \eqn{strong stability} to hold, 
and recall the definition \eqn{eta} of $\eta^{(i)}_{\vc{c}}$. Let $x\ge 0$ and $\vc{c} \ge \vc{0}$ be a directional vector. For the case (C0),
\begin{eqnarray}
\label{eqn:bound 0a 1}
  \dd{P}(S_{\vc{c}}^{I(r_{\vc{c}})} > x) &\le& \dd{P}_{\pi}(c_{1}Z_{1} + c_{2} Z_{2} > x) \nonumber\\
  &\le& \delta_{1} \eta^{(1)}_{\vc{c}} \dd{P}(c_{2} S_{2}^{I(r_{2})} + S_{\vc{c}}^{I(r_{\vc{c}})} > x) + \delta_{2} \eta^{(1)}_{\vc{c}}\dd{P}(c_{1} S_{1}^{I(r_{1})} + S_{\vc{c}}^{I(r_{\vc{c}})} > x). \qquad
\end{eqnarray}
For the case (C1),
\begin{eqnarray}
\label{eqn:bound 0a 2}
  \dd{P}(S_{\vc{c}}^{I(r'_{\vc{c}})} > x) \le \dd{P}_{\pi}(c_{1}Z_{1} + c_{2} Z_{2} > x) \le \delta_{1} \eta^{(1)}_{\vc{c}} \dd{P}(c_{2} S_{2}^{I(r_{2})} + S_{\vc{c}}^{I(r_{\vc{c}})} > x).
\end{eqnarray}
Here random variables  $S_{1}^{I(r_{1})}$ and $S_{2}^{I(r_{2})}$ are assumed to be
independent of $S_{\vc{c}}^{I(r_{\vc{c}})}$.
\end{theorem}
\begin{remark} {\rm
\label{rem:bound 0a}
  For the case (C0), $\vc{c} > \vc{0}$, so \eqn{bound 0a 1} does not contradict \eqn{tail Z1a}.
}\end{remark}

\begin{proof}
 In the case (C0), by \lem{boundary measures 2}, \thr{option 0 1} and its symmetric version, 
\begin{eqnarray*}
 && \varphi_{3-i}(0) \dd{P}(V_{i} > x) \le \delta_{3-i} \dd{P}(Z_{i} > x) \le \delta_{3-i} \dd{P}(S_{i}^{I(r_{i})} > x), \qquad i=1,2.
\end{eqnarray*}
Hence, \lem{decomposition 1} yields  the upper bound of \eqn{bound 0a 1}, and its lower bound is  obtained by \cor{lower bound 1}. In the case (C1), the upper bound of \eqn{bound 0a 2} is immediate from \lem{decomposition 1}, while the lower bound is already obtained in \lem{decomposition 2}.
\end{proof}

Assume now that the three distributions, $F_{1}^I$, $F_{2}^I$ and $F_{\vc{c}}^I$
are all subexponential. Then
$$
\dd{P}(c_{i} S_{i}^{I(r_{i})}>x)\sim \frac{r_{1}}{1-r_{1}}\overline{F}_{i}^I(x/c_{i}), \quad i=1,2,
\qquad
\dd{P}(S_{\vc{c}}^{I(r_{\vc{c}})}>x)\sim 
\frac{r_{\vc{c}}}{1-r_{\vc{c}}}\overline{F}_{\vc{c}}^I(x).
$$
Similar asymptotic equivalences hold when one replaces
$r_{\vc{c}}$ by either $r^{'}_{\vc{c}}$ or $r^{''}_{\vc{c}}$.
Clearly,
\begin{eqnarray*}
  \ol{F}^{I}_{\vc{c}}(x) = \frac{1}{m_{\vc{c}}} \int_x^{\infty} \dd{P} (c_{1} J_{1} + c_{2} J_{2}>y) dy \ge \frac{c_{1}}{m_{\vc{c}}} \int_{x/c_{1}}^{\infty} \dd{P} (J_{1}>y) dy = \frac{c_{1}m_{1}}{m_{\vc{c}}} \ol{F}^{I}_{1}(x/c_{1}),
\end{eqnarray*}
therefore,
$$
\dd{P}(c_{1}S_{1}^{I(r_{1})}>x) \le (1+o(1)) \frac{m_{\vc{c}}r_{1}}{(1-r_{1})c_{1}m_{1}}
\overline{F}_{\vc{c}}^I(x) \le (1+o(1))
\frac{(1-r_{\vc{c}})m_{\vc{c}}r_{1}}{r_{\vc{c}}(1-r_{1})c_{1}m_{1}}
\dd{P}(S_{\vc{c}}^{I(r_{\vc{c}})}>x)
$$
and, finally,
\begin{eqnarray*}
  \dd{P}(c_{1} S_{1}^{I(r_{1})} + S_{\vc{c}}^{I(r_{\vc{c}})} > x) \le
  (1+o(1))
  K\ol{F}^{I}_{\vc{c}}(x).
\end{eqnarray*}
Here constant 
$
K := ({(1-r_{\vc{c}})m_{\vc{c}}r_{1}})/({r_{\vc{c}}(1-r_{1})c_{1}m_{1}})
+ ({r_{\vc{c}}})/({1-r_{\vc{c}}})
$
is positive and finite. These observations and  \thr{bound 0a} lead to  the following result.

\begin{corollary}
\label{cor:bound 0a}
  In the both cases of \thr{bound 0a}, if the three
  distributions  $F_{1}^I$, $F_{2}^I$ and $F_{\vc{c}}^I$ are subexponential, 
  then the distribution of $c_{1}Z_{1} + c_{2} Z_{2}$ is weakly 
  tail-equivalent to $F^{I}_{\vc{c}}$.
\end{corollary}

\begin{remark}
\label{rem:}
Subexponentiality of both $F_{1}^I$ and $F_{2}^I$ may not imply that of
$F_{\vc{c}}^I$, even if jump sizes $J_{1}$ and $J_{2}$ are independent. It holds 
if the distributions are regularly varying. 
In general, necessary and sufficient conditions for subexponentiality of 
the convolution 
of subexponential distributions
is given in \citet{EmbrGold1980} (see also Theorem 3.33 in
\cite{FossKorsZach2011} for complete characterization).
\end{remark}

\section{Sample-path approach} \setnewcounter
\label{sect:sample}

Throughout the section, we assume that
\begin{itemize}
\item [(A2)] The point process $\{N(t)\}$ is a renewal process with i.i.d. inter-arrival times having a general distribution with finite mean $a \equiv 1/\lambda$,   and jumps are one-dimensional, that is, for $p_{1}, p_{2} > 0$ satisfying $p_{1}+p_{2}=1$,
\begin{eqnarray}
\label{eqn:independent jump}
  F(x,y) = p_{1} F_{1}(x) + p_{2} F_{2}(y), \qquad x, y \ge 0.
\end{eqnarray}
\end{itemize}

Based on fluid dynamics considered in \app{pure fluid}, we provide a lower bound for the tail probabilities assuming only that the integrated tail distributions $F_{1}^I$ and $F_{2}^I$ 
are {\it long-tailed} (see \app{heavy tail} for definitions). Then we get the exact asymptotics
in the case of subexponential distribution.

\noindent {\bf Lower bound}.
In what follows, we use notation $LB(x)$ for the lower bound for the probability
$\dd{P} (c_{1}Z_+c_{2}Z_{2}>x)$. By \cor{lower bound}, for $c_{1},c_{2}\ge 0$ with $c_{1}+c_{2}=1$,  
\begin{eqnarray}
\label{eqn:LB 1}
\lefteqn{LB(x) = (1+o(1))\sum_{n=1}^{\infty} \Big(
p_{1} \dd{P}\big(J_{1,-n}> x/c_{1}+na(\Delta_{1}+p_{21}\Delta_{2})\big)} \nonumber \hspace{18ex}\\
&& \qquad +p_{2} \dd{P}\big(J_{2,-n}> x/c_{2}+na(\Delta_{2}+p_{12}\Delta_{1})\big)\Big), \quad x > 0.
\end{eqnarray}
Here, we have used the following heuristics: the lower bound asymptotics in our model is equivalent to that in an auxiliary
fluid model where (i) possible time instants of big jumps are $-na$, $n=1,2,\ldots$,
and (ii) all further input jumps (after the big one) are replaced by continuous fluid inputs 
with constant rates $\alpha_{1}$ and $\alpha_{2}$, respectively.
So the lower bound constitutes a probability of a union of infinite number of trajectories,
with a single big random jump followed by a continuous path. 
Since the probability
of simultaneous occurrence of two or more big jumps is negligibly small, 
the probability of a union of events may be replaced by the sum of probabilities. These heuristic arguments can be justified by exact and complete mathematical calculations and statements along the lines of, say, 
\cite{BaccFoss2004} or/and \cite{FossKors2006}.  

Assume both $F_{1}^I$ and $F_{2}^I$ to be long-tailed, then \eqn{LB 1} yields
\begin{eqnarray}
\label{eqn:LB 2}
  LB(x) = \label{ququ}
(1+o(1))\frac{\alpha_{1}}{\Delta_{1}+p_{21}\Delta_{2}}\overline{F}_{1}^I(x/c_{1})
+(1+o(1))\frac{\alpha_{2}}{\Delta_{2}+p_{12}\Delta_{1}}\overline{F}_{2}^I(x/c_{2}),
\end{eqnarray}
where, by convention, $x/0 = \infty$ and $\overline{F}_i^I(\infty )=0$. In particular, for $c_{1}=1, c_{2}=0$,
\begin{eqnarray*}
   \frac{r'_{1}}{1-r'_{1}} = \frac {\alpha_{1}} {\mu_{1}(1-p_{12} p_{21}) - \alpha_{1}} < \frac{\alpha_{1}}{\Delta_{1}+p_{21}\Delta_{2}} < \frac {\alpha_{1}} {\Delta_{1}}.
\end{eqnarray*}
Hence, \eqn{LB 2} gives a better (tighter) bound than the one in \eqn{tail Z1a} (see \eqn{SI asymptotics}). 

\noindent {\bf Tail equivalence.}  We refer to \lem{upper_parallel}.
It is known that, for a single-node queue $\tilde{Z}_{i}$ with subexponential
batch distributions, the stationary content is large due to a single large value of
one of batches, and the tail distribution of the stationary content is equivalent
to the integrated tail $\overline{F}^I_i(x)$.
Then, by \lem{upper_parallel}
and by lower bound \eqref{ququ} with $c_1=1$, the tail asymptotics for the stationary $Z_{i}$ is
weakly tail-equivalent to $\overline{F}^I_i(x)$. Similar result holds for the 
linear sum $c_{1}Z_{1}+c_{2}Z_{2}$ if we assume the tails $\overline{F}_1(x)$ and $\overline{F}_2$ to be equivalent. But here we can get more.

\noindent {\bf Exact asymptotics.}  
We first consider exact tail asymptotics for $c_{1} = 1, c_{2} = 0$, with applying the ``squeeze principle'' (see Theorem 8 in  \cite{BaccFoss2004}). Thus, we focus on the tail asymptotic for $Z_{1}$. This is done by the following two steps.

\noindent ({\it 1st Step}) Assume that the distributions $F_{1}^I$ are subexponential.
Assume that both the systems run in the stationary regime and let $Z_{1} = Z_{1}(0)$ and $\tilde{Z}_{1} = \tilde{Z}_{1}(0)$. Following the lines of Theorem $5.4^*$ in \cite{FossKorsZach2011}, one can show that, for $i=1,2$,
\begin{eqnarray*}
  \dd{P} (\tilde{Z}_{1}>x) = (1+o(1)) \sum_{n=0}^{\infty} p_{1} \dd{P} (\tilde{Z}_{1} >x, J_{i, (-n)}>x+an\Delta_i), \quad x\to\infty,
\end{eqnarray*}
Since $Z_{1}\le \tilde{Z}_{1}$ a.s., this and the fact that $\liminf_{x\to\infty} \dd{P} (Z_{1}>x)/\dd{P}(\tilde{Z}_{1}>x) >0$ yields
\begin{eqnarray}
\label{eqn:Z sum}
  \dd{P} (Z_{1}>x) \!\! &=& \!\! \dd{P} (Z_{1}>x, \tilde{Z}_{1}>x) = \sum_{n=0}^{\infty} p_{1} \dd{P} (Z_{1} >x, J_{1,(-n)}>x+an\Delta_{1}) + o(\dd{P} (\tilde{Z}_{1} >x)) \nonumber\\
&=& \!\! (1+o(1)) \sum_{n=0}^{\infty} p_{1} \dd{P} (Z_{1}>x, J_{i,(-n)}>x+an\Delta_{1}), \quad x\to\infty,
\end{eqnarray}
\noindent ({\it 2nd Step}) If there is only one big jump before time 0 and the other jumps are uniformly approximated by a fluid limit, then it follows from \lem{lower bound} that $Z_{1}(0) > x$ occurs only if there is a  $(-n)$-th arrival, whose arrival time is denoted by $T_{-n} < 0$, such that $J_{1,(-n)} > x + (-T_{-n}) (\Delta_{1} + p_{21} \Delta_{2})$. Applying similar arguments to those in the proof of Theorem 8 on \cite{BaccFoss2004} with help of \eqn{Z sum}, we can get that
\begin{eqnarray*}
  \dd{P} (Z_{1}>x) = (1+o(1)) \sum_{n=0}^{\infty} p_{1} \dd{P} (
  J_{1,(-n)} > x + an (\Delta_{1} + p_{21} \Delta_{2})), \quad x\to\infty.
\end{eqnarray*}
Hence, we have the upper bound for $\dd{P} (Z_{1}>x)$ which coincides with the lower bound.

By symmetry, we have a similar result for $c_{1} = 0, c_{2} = 1$. We finally assume that $c_{1} > 0, c_{2} > 0$. If in addition 
$F_1^I(x/c_1)$ and $F_2^I(x/c_2)$ are weakly tail-equivalent, then, due to Theorem 3.33
from \cite{FossKorsZach2011}, the distribution of
$c_1\widetilde{Z}_1+c_2\widetilde{Z}_2$ is also subexponential and
$$
\dd{P} (c_{1}\widetilde{Z}_{1}+c_{2}\widetilde{Z}_{2}>x) 
= (1+o(1))  \left(\dd{P} (\widetilde{Z}_{1} >x/c_{1})
+\dd{P}(\widetilde{Z}_{2}>x/c_{2})\right)
$$
where we again use convention $x/c_i=\infty$ if $c_i=0$. 
Then, similarly,
\begin{eqnarray*}
\dd{P} (c_{1}Z_{1}+c_{2}Z_{2}>x) &=& (1+o(1))
\sum_{n=0}^{\infty} p_{1} \dd{P} (J_{1,(-n)} > x/c_1 + an (\Delta_{1} + p_{21} \Delta_{2}))\\
&& + (1+o(1))
\sum_{n=0}^{\infty} p_{2} \dd{P} ( J_{2,(-n)} > x/c_2 + an (\Delta_{2} + p_{12} \Delta_{1})).
\end{eqnarray*}
Thus, we again have the upper bound which coincides with the lower bound \eqn{LB 1}, and there we have the following theorem.
\begin{theorem}
\label{thr:exact} 
Assume that (A2) holds, both $\Delta_{1}$ and $\Delta_{2}$ are positive,  
distributions $F_{1}^I(x/c_1)$ and $F_{2}^I(x/c_2)$ are both subexponential and  weakly tail-equivalent.
Then \eqn{LB 2} provides the exact asymptotics
for $\dd{P} (c_{1}Z_{1}+c_{2}Z_{2}>x)$. If, say, $c_{1}=1$, then the first term in \eqn{LB 2}
gives the exact asymptotics for $\dd{P}(Z_{1}>x)$.
 \end{theorem}
 
{\bf Acknowledgement.} S.Foss thanks Tokyo University of Science for hospitality during his visit, where this work was partly done by support of Japan Society for the Promotion of Science under grant No.\ 21510615. 

\appendix

\section*{Appendix}

\section{Analysis of a pure fluid model} \setnewcounter
\label{app:pure fluid}

Assume again $\Delta_{1}>0$ and $\Delta_{2}>0$.  
Consider an auxiliary pure fluid model with continuous fluid input rates
$\alpha_{1}$, $\alpha_{2}$, service rates $\mu_{1},\mu_{2}$ and
transition fractions $p_{12}, p_{21}$ as described in \sectn{fluid}. 
We use the same notation $Z_{1}(t), Z_{2}(t)$ as before, but for the deterministic
buffer quantities.

Let $t>0$ be fixed.
We assume that the fluid model starts at negative time $-t$ from levels
$y_{1},y_{2}$ (this means $Z_{1}(-t)=y_{1},Z_{2}(-t)=y_{2}$) and want to identify conditions 
on $y_{1},y_{2}$ for 
\begin{equation}\label{ccx}
c_{1}Z_{1}+c_{2}Z_{2}\ge x
\end{equation}
to hold where again $c_{1},c_{2}\ge 0$ and $c_{1}+c_{2}=1$ are given constants,
and where $Z_{i}=Z_{i}(0)$.



{\bf Case 1: $c_{1}=1,c_{2}=0$}.
Thus, we have to find conditions for $Z_{1}\ge x$. 

Due to monotonicity
properties of fluid limits (see, e.g., \citet{Lela2005a}), under the stability
conditions, if the fluid model starts from a non-zero initial value 
at time $-t$ and if some coordinate, say $i$, becomes zero, $Z_{i}(u)=0$ at time $u>-t$,
then it stays at zero, $Z_{i}(v)=0$ for all $u\le v\le 0$.

Let $L_{2}=y_{2}/\Delta_{2}$. Assume first that $L_{2}\ge t$. Then, at any time instant
$u\in (-t, 0)$,\\
(i) the input rate to Queue 1 is $\alpha_{1}+\mu_{2}p_{21}$; 
(ii) the output rate from Queue 1 is $\mu_{1}$; 
(iii) the input rate to Queue 2 is $\alpha_{2}+\mu_{1}p_{12}$; 
(iv) the output rate from Queue 2 is $\mu_{2}$.

Then $Z_{1}\ge x$ means that 
\begin{equation}\label{cond:op0.1_c20_ge}
y_{1}\ge x+t\Delta_{1}.
\end{equation}

 Assume now that $L_{2}<t$. Then, for any $u\in (-t,-t+L_{2})$,\\
(i) the input rate to Queue 1 is $\alpha_{1}+\mu_{2}p_{21}$; 
(ii) the output rate from Queue 1 is $\mu_{1}$;\\
(iii) the input rate to Queue 2 is $\alpha_{2}+\mu_{1}p_{12}$; 
(iv) the output rate from Queue 2 is $\mu_{2}$; \\
and, for any $u \in (-t+L_{2}, 0)$,\\
(a) the input and output rates to/from Queue 1 and the input rate to Queue 2 are as before, 
but 
(b) the output rate from Queue 2 equals to the input rate, i.e. is $\alpha_{2}+\mu_{1}p_{12}$.

Then condition $Z_{1}\ge x$ is equivalent to
\begin{equation}\label{cond:op0.1_c20_le}
y_{1} - x \ge 
L_{2} \Delta_{1} + (t-L_{2}) (\mu_{1} - (\alpha_{1} + p_{21}(\alpha_{2}+\mu_{1}p_{12}) ))
= t\Delta_{1} +tp_{21}\Delta_{2} - y_{2}p_{21}.
\end{equation}
Combining \eqref{cond:op0.1_c20_ge} and \eqref{cond:op0.1_c20_le} together, we
have 
\begin{equation}\label{cond:op0.1_c20}
y_{1}\ge x+t\Delta_{1} + p_{21}(t\Delta_{2}-y_{2})^+.
\end{equation}

{\bf Case 2: $c_{1}=0,c_{2}=1$}.
This case is symmetric to the previous one.

{\bf Case 3: $c_{1}>0,c_{2}>0$}.
Following the same logics as before, we get:\\
if $L_{2}\le t$, then $Z_{2}=0$, and the condition on $y_{1}$ coincides
with \eqref{cond:op0.1_c20_le} if one replaces $x$ by $x/c_{1}$. More precisely,
we get inequality:
\begin{equation}\label{cond:op0.3_ge_le1}
y_{1} \ge 
x/c_{1}+t\Delta_{1} +p_{21}(\Delta_{2} - y_{2}).
\end{equation}

Similarly, if $L_{1} \le t$, then $Z_{1}=0$, and we get 
\begin{equation}\label{cond:op0.3_ge_le2}
y_{2}\ge x/c_{2} +t\Delta_{2} + p_{12}(t\Delta_{1}-y_{1}).
\end{equation}

Otherwise, if both $L_{1}> t$ and $L_{2}> t$, then $y_{1}=Z_{1}+t\Delta_{1}$,
$y_{2}=Z_{2}+t\Delta_{2}$, and we have 
\begin{equation}\label{cond:op0.3_ge_ge}
c_{1}Z_{1}+c_{2}Z_{2}=
c_{1}(y_{1}-t\Delta_{1})+c_{2}(y_{2}-t\Delta_{2})\ge x.
\end{equation}

Combining all three sub-cases, we arrive at the following result:

\begin{lemma}
\label{lem:lower bound}
Consider a purely fluid model with 
input rates
$\alpha_{1}$, $\alpha_{2}$, service rates $\mu_{1},\mu_{2}$ and
transition fractions $p_{12}>0, p_{21}>0$.
Let $c_{1},c_{2}$ be non-negative constants with $c_{1}+c_{2}=1$.
Let $t>0$ and let the system start at time $-t$
from $Z_{1}(-t)=y_{1}\ge 0$ and
$Z_{2}(-t)=y_{2}\ge 0$.
If $\Delta_{1}>0$ and $\Delta_{2}>0$, then inequality 
$c_{1}Z_{1}+c_{2}Z_{2}\ge x$ holds if and only if 
\begin{equation}\label{cond:op0}
c_{1} \left(y_{1}-t\Delta_{1}-p_{21}(t\Delta_{2}-y_{2})^+\right)^+
+
c_{2} \left(y_{2}-t\Delta_{2}-p_{12}(t\Delta_{1}-y_{1})^+\right)^+
\ge 
x.
\end{equation}
\end{lemma}

\begin{corollary}
\label{cor:lower bound}
 Consider a particular case where only one of two options
 is possible, either $y_{1}>0$ and $y_{2}=0$ or $y_{1}=0$ and $y_{2}>0$.
 Then \eqref{cond:op0} is equivalent to
 $$
 \max (c_{1}(y_{1}-t\Delta_{1}-p_{21}t\Delta_{2}), c_{2}(y_{2}-t\Delta_{2}-p_{12}
 t\Delta_{1}))>x.
 $$
In turn, the latter inequality is
 equivalent to a union of two events,
 \begin{equation}\label{1001}
 \{y_{1}>x/c_{1}+t\Delta_{1}+p_{21}t\Delta_{2}\} \cup
 \{y_{2}>x/c_{2}+t\Delta_{2}+p_{12}t\Delta_{1}\},
 \end{equation}
where one of these events is empty if the corresponding $c_i$ equals 0.
\end{corollary}

\section{Heavy-tailed distributions}
\label{app:heavy tail}

{\it Definitions:} Distribution $F$ of a positive random variable $X$
is  
(1) {\it heavy-tailed} if ${\mathbf E} e^{cX} \equiv
\int_0^{\infty} e^{cx} dF(x)=\infty$ 
for all $c>0$; and {\it light-tailed} otherwise;
(2) {\it long-tailed} if $\overline{F}(x)>0$ for all $x>0$ and
$\overline{F}(x+1)/\overline{F}(x)\to 1$ as $x\to\infty$;
(3) {\it subexponential} if $\overline{F*F}(x) \sim 2\overline{F}(x)$
or, equivalently, $\dd{P}(X_{1}+X_{2}>x)\sim 2\dd{P}(X>x)$ as $x\to
\infty$ (here $X_{1}$ and $X_{2}$ are two independent copies of $X$).
(4) Distribution $F$ of a real-valued r.v. $X$ is subexponential if
the distribution of $\max (X,0)$ is subexponential.
(5) Distribution $F$ is {\it regularly varying} if $\overline{F}(s) = l(x) x^{-k}$,
where $k\ge 0$ and positive function $l(x)$ is {\it slowly varying} at infinity. Regularly
varying distributions are subexponential.

{\it Key Properties:}
(1) Any subexponential distribution is long-tailed,
and any long-tailed distribution is heavy-tailed. 
(2) If distribution $F$ is long-tailed, then there
exists a function $h(x)\to\infty$ as $x\to\infty$ such that
$\overline{F}(x+h(x))/\overline{F}(x)\to 1$ as $x\to\infty$.
(3) If $F$ is subexponential and $\overline{G}(x)\sim
\overline{F}(x)$, then $G$ is subexponential. 
(4) If $X_{1},X_{2},\ldots$ are i.i.d. with common subexponential distribution $F$ and if $\tau$ is a light-tailed counting r.v.,
then $\sum_{1}^{\tau} X_i$ also has a subexponential distribution and 
$ \dd{P} 
( \sum_{1}^{\tau} X_i >x
)\sim {\mathbf E} \tau
\overline{F}(x).
$

%

\end{document}